\documentclass[12pt]{article}
\usepackage{amsmath,amsthm}
\usepackage{amssymb}
\usepackage[pdftex]{graphicx}
\usepackage[T1]{fontenc}

\newtheorem{theorem}{Theorem}[section]
\newtheorem{corollary}[theorem]{Corollary}

\newtheorem{proposition}[theorem]{Proposition}
\newtheorem{problem}[theorem]{Problem}

\theoremstyle{definition}

\newtheorem{step}{Step}

\numberwithin{equation}{section}

\newcommand{\Set}[2]{
\left\{#1\,\middle|\,#2\right\}}

\newcommand{\R}{\mathbb{R}}
\newcommand{\M}{\mathcal{M}}

\newcommand{\B}{\mathcal{B}}

\newcommand{\X}{\mathcal{X}}
\newcommand{\Y}{\mathcal{Y}}
\newcommand{\dGH}{d_{\mathrm{GH}}}
\newcommand{\dis}{\mathop{\mathrm{dis}}\nolimits}
\newcommand{\diam}{\mathop{\mathrm{diam}}\nolimits}
\newcommand{\GHC}{\mathop{\mathrm{GHC}}\nolimits}
\newcommand{\Sum}{\mathop{\mathfrak{S}}\nolimits}

\begin{document}

\title{An Isometric Embedding of the $\ell^\infty$ product space of two bounded subspaces of the Gromov-Hausdorff Space into the Gromov-Hausdorff Space}

\author{Takuma Byakuno
\footnote{Institute of Mathematics\\ 
Graduate School of Science and Engineering, Kansai University\\
3-3-35 Suita, 564-8680 Osaka, Japan\\
E-mail: k241344@kansai-u.ac.jp}}

\date{}

\maketitle


\renewcommand{\thefootnote}{}

\footnote{2020 \emph{Mathematics Subject Classification}: Primary 30L05; Secondary 54E35, 53C23.}

\footnote{\emph{Key words and phrases}: Gromov-Hausdorff space, Isometric embeddings.}

\renewcommand{\thefootnote}{\arabic{footnote}}
\setcounter{footnote}{0}


\begin{abstract}
In this paper, 
we prove the $\ell^\infty$ product space of two bounded subspaces of the Gromov-Hausdorff space can be isometrically embedded into the Gromov-Hausdorff space.
\end{abstract}

\section{Introduction}
In 1909, Fr\'{e}chet \cite[pp.161--162]{Frechet} showed that an arbitrary separable metric space can be isometrically embedded into the metric space of all real-valued and bounded sequences $(x_n)_{n=1}^\infty=(x_1,x_2,\ldots,x_n,\ldots)$, which is usually denoted by $\ell^\infty$ and equipped with the supremum distance:
\[
\mathrm{the\ distance\ between}\ (x_n)_{n=1}^\infty\ \mathrm{and}\ (y_n)_{n=1}^\infty:
\sup_{n=1,2,\ldots}|x_n-y_n|.
\]
Indeed, we can define the isometric embedding of a separable metric space $(X,d)$ into $\ell^\infty$ by $x\mapsto (d(x,x_n)-d(x_n,x_1))_{n=1}^\infty$ for $x\in X$, where $\{x_1,x_2,\ldots\}$ is a dense subset of $X$.\\

In 1927, Urysohn \cite{Urysohn} constructed the separable complete metric space $U$ satisfying the following:
\begin{itemize}
\item
For an arbitrary separable metric space $X$, there exists a suitable subspace $M$ of $U$ such that $M$ is isometric to $X$,
\item
An arbitrary isometry between finite subsets of $U$ can be extended to an isometry on $U$.
\end{itemize}
Additionally, $U$ is unique up to an isomorphism and called
the {\it Urysohn universal space} today.\\

In 1932, Banach \cite[p.187]{Banach} showed that an arbitrary separable metric space can be isometrically embedded into the metric space of all continuous functions on the closed interval $[0,1]$. 
It is a Banach space usually denoted by $C([0,1])$ and equipped with the supremum distance:
\[
\mathrm{the\ distance\ between}\ f\ \mathrm{and}\ g:
\sup_{0\leq t\leq 1}|f(t)-g(t)|.
\]
This fact was proved as an application of Fr\'{e}chet's isometric embedding and Banach-Mazur theorem \cite[p.185]{Banach}, which states that every Banach space can be isometrically embedded into $C([0,1])$.
This space $C([0,1])$ is not isometric to the Urysohn universal space according to Mazur-Ulam theorem, which states that all surjective isometries between normed spaces are affine. 
However, in 2008, Dutrieux and Lancien \cite{Dutrieux} showed that a Banach space into which all compact metric spaces can be isometrically embedded has a linearly isometric copy of $C([0,1])$.\\

Generally, a metric space $X$ is {\it universal} for a class $C$ consisting of metric spaces if all $Y\in C$ can be isometrically embedded into $X$. 
The above works focus on the universality for a class of all separable metric spaces.\\

On the other hand, in 2017, Iliadis, Ivanov, and Tuzhilin \cite{Tuzhilin1} showed that all finite metric spaces can be isometrically embedded into the Gromov-Hausdorff space, which means this space is universal for a class of all finite metric spaces.
The {\it Gromov-Hausdorff space} is a metric space consisting of a certain set $\M$ and Gromov-Hausdorff distance $\dGH$ defined later.
The set $\M$ is usually defined as the set of all isometry classes of  compact metric spaces, but we now consider $\M$ as a set characterized by
\begin{itemize}
\item
For every $X\in\M$, $X$ is a (nonempty) compact metric space,
\item
For an arbitrary compact metric space $X$,
we can find a suitable $Y\in\M$ such that $Y$ is isometric to $X$,
\item 
For every $X,Y\in\M$, we obtain $X=Y$ if $X$ is isometric to $Y$.
\end{itemize}
The Gromov-Hausdorff space $\M$ has a dense countable set consisting of isometric copies of all finite metric spaces endowed with rational valued distance functions and therefore $\M$ is separable. 
In addition, we can show that $\M$ is complete, but $\M$ is not isometric to the Urysohn universal space. 
See \cite[Section 3]{Tuzhilin1} and \cite[Main Theorem]{Tuzhilin2}.\\

According to \cite{Tuzhilin1}, there are many open problems related to the geometrical properties of $\M$. 
For instance, in this paper, we focus on
\begin{problem}
Can we isometrically embed all compact metric spaces into the Gromov-Hausdorff space $\M$? 
\end{problem}
Iliadis, Ivanov, and Tuzhilin's result can be considered as a partial answer to this problem, and this partial answer can be generalized to the following as an easy application of \cite[Theorem 4.1]{Tuzhilin1}.
\begin{proposition}
\label{EBs2GHs}
All bounded metric subspaces of an $\ell^\infty$ normed space $\R^n$ can be isometrically embedded into the Gromov-Hausdorff space. 
\end{proposition}

In this paper, we will prove the next theorem.
\begin{theorem}
\label{theorem1}
Let $\M_{\leq r}, \B(r)$ $(r>0)$ be the following metric subspaces of the Gromov-Hausdorff space $\M${$\rm{:}$}
\[
\M_{\leq r}=\Set{Z\in\M}{\diam(Z)\leq r}\,,\,
\B(r)=\Set{Z\in\M}{\diam(Z)=r}\,,
\]
where $\diam(Z)$ is the diameter of a metric space $Z$.
Then the $n$th cartesian power $(\M_{\leq r})^n$ equipped with $\ell^\infty$ distance can be embedded into $\B(5rn)$.
Namely, we can construct a map $\Sum_{r,n}:(\M_{\leq r})^n\to\B(5rn)$ satisfying
\[
\dGH(\Sum_{r,n}(X),\Sum_{r,n}(Y))=\max_{k=1,2,\ldots,n}\dGH(X_k,Y_k)
\]
for all $X=(X_k)_{k=1}^n, Y=(Y_k)_{k=1}^n\in(\M_{\leq r})^n$.
\end{theorem}
This theorem implies
\begin{corollary}
\label{theorem2}
Let $\X, \Y$ be bounded metric subspaces of the Gromov-Hausdorff space. 
Then the product space $\X\times\Y$ equipped with $\ell^\infty$ distance can be isometrically embedded into the Gromov-Hausdorff space.
\end{corollary}
\begin{corollary}
\label{theorem3}
Let $X, Y$ be bounded metric spaces which can be isometrically embedded into the Gromov-Hausdorff space.
Then the product space $X\times Y$ equipped with $\ell^\infty$ distance can be isometrically embedded into the Gromov-Hausdorff space.
\end{corollary}
In particular, we can also show Proposition \ref{EBs2GHs} by using this corollary since a bounded subspace of $\R^1$ can be isometrically embedded into $\M$.

\section{Preliminaries}
Let $X,Y$ be sets, and $R$ a subset of $X\times Y$.
For $A\subset X$, we define $R[A]$ as an {\it image} of $A$:
\[
R[A]=\Set{y\in Y}{\mathrm{there\ exists}\ a\in A\ \mathrm{such\ that}\ (a,y)\in R },
\]
and we define $R^{-1}$ as an {\it inverse} of $R$:
\[
R^{-1}=\Set{(y,x)}{(x,y)\in R}\subset Y\times X.
\]
Thus, an {\it inverse image} $R^{-1}[B]$ is also defined similarly.\\

Using this notation, we write $\GHC(X,Y)$ for the system
of all subsets $R$ of $X\times Y$ satisfy $R[\{x\}]\neq\emptyset$ and $R^{-1}[\{y\}]\neq\emptyset$ for all $x\in X$ and $y\in Y$.
In addition, we define $\dis(R)$ as the {\it distortion} of $R\in\GHC(X,Y)$ if $X, Y$ are metric spaces endowed with distance functions $d_X,d_Y$ respectively:
\[
\dis(R)=\dis(X,Y;R)=\sup\Set{|d_X(x,\xi)-d_Y(y,\eta)|}{(x,y),(\xi,\eta)\in R}.
\]

As described above, the Gromov-Hausdorff space is a metric space consisting of a certain set $\M$ and Gromov-Hausdorff distance $\dGH$.
The Gromov-Hausdorff distance is typically defined by using the ``Hausdorff distance'', 
but we can alternatively define the {\it Gromov-Hausdorff distance} between $X,Y\in \M$ as 
\[
\dGH(X,Y)=\frac12\inf_{R\in\GHC(X,Y)}\dis(R)\,.
\]

We define $\diam(X)$ as the {\it diameter} of a metric space $(X,d)$:
\[
\diam(X)=\diam(X,d)=\sup_{x_0,x_1\in X}d(x_0,x_1).
\]
It is straightforward to see $\dGH(X,Y)\leq\max\{\diam(X),\diam(Y)\}$ and
\[
|\diam(X)-\diam(Y)|\leq 2\dGH(X,Y)
\]
for $X,Y\in\M$, which means a map $\diam:\M\to\R$ is 2-lipschitz continuous.
Therefore, if $\X$ is a bounded set of $\M$, we can obtain $\sup_{X\in\X} \diam(X)<\infty$.\\

By the definition of $\M$, for an arbitrary compact metric space $X$, there exists uniquely $Y\in\M$ such that $X$ is isometric to $Y$.
Then we may consider $[X]=Y$.
On the other hand, we can similarly define the Gromov-Hausdorff distance between two compact metric spaces $X, Y$ which do not necessarily satisfy $X, Y\in\M$.
However, we can see easily $\dGH(X,Y)=\dGH([X],[Y])$.
Generally speaking, we can see $\dGH(X, Y)=\dGH(Z, Y)$ if $X$ is isometric to a metric space $Z$.

\section{Proof of Theorem \ref{theorem1}}
We shall prove Theorem \ref{theorem1} and show some corollaries.

\begin{proof}[Proof of Theorem \ref{theorem1}]
First of all, we define a compact metric space $\Sum_{r,n}(X)$ for an arbitrary point $X=(X_k)_{k=1}^n\in(\M_{\leq r})^n$.\\

we fix a two-point set $X_0=\{p^1,p^{-1}\}$ equipped with distance function $d_0$ satisfying $d_0(p^+,p^-)=3r$ and we may assume $X_0,X_1,\ldots,X_n$ are disjoint by replacing isometric copies of them.
For simplicity, we consider $p^1=-p^{-1}$ and $p^{-1}=-p^1$.\\

Then we put $\Sum_{r,n}(X)=\bigcup_{k=0}^n X_k$ as a set and we can define the distance function $d^{r,n}_X$ of $\Sum_{r,n}(X)$ satisfying the following if the distance function of $X_k$ is denoted by $d_k$:
\[
d^{r,n}_X(x_k,x_l)
=\begin{cases}
d_k(x_k,x_l)&\mathrm{if:}k=l,\\
5r|l-k|&\mathrm{if:}k\neq l
\end{cases}
\]
for $x_k\in X_k$, $x_l\in X_l$, $k,l\in\{0,1,\ldots,n\}$.
\begin{figure}[h]
\centering
\includegraphics[width=1\textwidth]{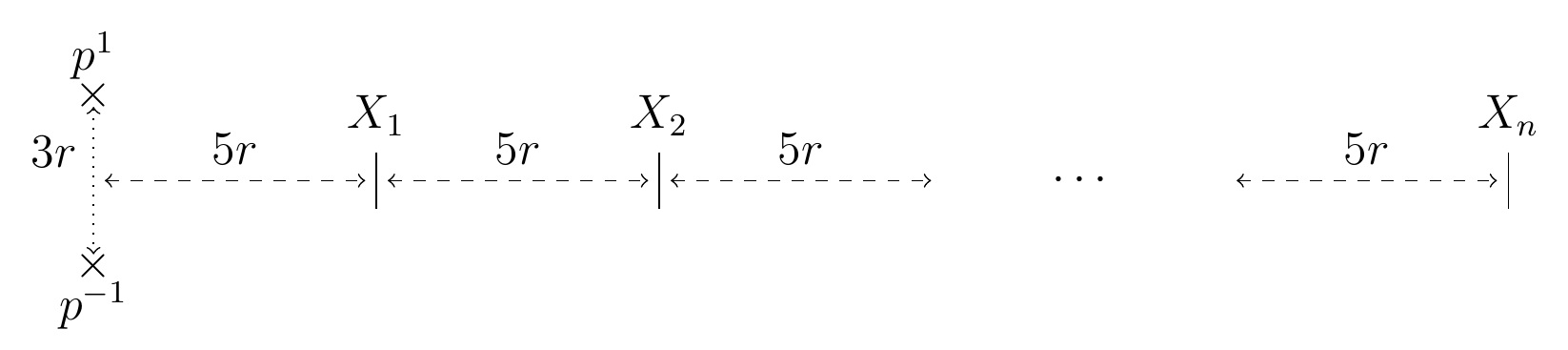}
\caption{The definition of $\Sum_{r,n}(X)$}
\end{figure}
\\

Now we consider $\Sum_{r,n}(X)=[\Sum_{r,n}(X)]\in\M$ for simplicity.
In particular, by mean of $d^{r,n}_X(p^1,x_n)=5rn$ for some point $x_n\in X_n$, we have $\diam(\Sum_{r,n}(X))=5rn$, which means $\Sum_{r,n}(X)\in\B(5rn)$.
Thus, we can define a map $\Sum_{r,n}:(\M_{\leq r})^n\to\B(5rn)$.\\

Let $X=(X_k)_{k=1}^n,Y=(Y_k)_{k=1}^n$ be arbitrary fixed points of $(\M_{\leq r})^n$. 
we put $m=\max_{k=1,2,\ldots,n}\dGH(X_k,Y_k)\leq r$ for short.
We shall prove 
\[
\dGH(\Sum_{r,n}(X),\Sum_{r,n}(Y))=m=\max_{k=1,2,\ldots,n}\dGH(X_k,Y_k).
\]

Firstly, we shall show $\dGH(\Sum_{r,n}(X),\Sum_{r,n}(Y))\leq m$.
For any positive number $\varepsilon$ with $m<\varepsilon$, 
we can obtain $R_k\in\GHC(X_k,Y_k)$ $(k=1,2,\ldots,n)$ satisfying $\dis(R_k)<2\varepsilon$ by the definition of the Gromov-Hausdorff distance.
Then the following set $R$ belongs to $\GHC(\Sum_{r,n}(X),\Sum_{r,n}(Y))$:
\[
R=R_0\cup\bigcup_{k=1}^nR_k\subset\Sum_{r,n}(X)\times\Sum_{r,n}(Y)\,,
\]
where $R_0=\{(p^1,p^1),(p^{-1},p^{-1})\}\subset X_0\times Y_0$ (we put $Y_0=\{p^1,p^{-1}\}$).
It is straightforward to see $\dis(R)<2\varepsilon$, which means $\dGH(\Sum_{r,n}(X),\Sum_{r,n}(Y))<\varepsilon$. 
Indeed, for any points $(x,y)\in R_k$ and $(\xi,\eta)\in R_l$, we obtain
\[
|d^{r,n}_X(x,\xi)-d^{r,n}_Y(y,\eta)|
\leq
\begin{cases}
0&\mathrm{if:}k\neq l,\\
\dis(R_k)&\mathrm{if:}k=l
\end{cases}\ (\dis(R_0)=0)\,.
\]
As a consequence, we have $\dGH(\Sum_{r,n}(X),\Sum_{r,n}(Y))\leq m$.\\

Secondary, we shall show $\dGH(\Sum_{r,n}(X),\Sum_{r,n}(Y))\geq m$ by contradiction.
Thus, we suppose $\dGH(\Sum_{r,n}(X),\Sum_{r,n}(Y))<m$, which means there exists $R\in\GHC(\Sum_{r,n}(X),\Sum_{r,n}(Y))$ such that $\dis(R)<2m\leq 2r$.
Now we divide the argument among $4$ steps to obtain a contradiction.

\begin{step}
In this step, we shall show that $4$ sets $R[\{p^1\}]$, 
$R[\{p^{-1}\}]$, $R^{-1}[\{p^1\}]$, $R^{-1}[\{p^{-1}\}]$ are one-point sets including in $X_0=Y_0=\{p^1,p^{-1}\}$, which means we can define $R^s(p^t)\in\{p^1,p^{-1}\}$ ($s,t\in\{1,-1\}$, $R^1=R$) as the point satisfying $R^s[\{p^t\}]=\{R^s(p^t)\}$.
In addition, we shall show that they satisfy
\[
R(p^1)=R^{-1}(p^1)
\,,\,
R(p^{-1})=R^{-1}(p^{-1})
\,,\,
R(p^1)=-R(p^{-1})
\,.
\]
 
Firstly, suppose, for the sake of contradiction, that $R[\{p^1\}]\not\subset\{p^1,p^{-1}\}$.
It means there exist a number $k=1,2,\ldots,n$ and a point $y_k\in Y_k$ such that $(p^1,y_k)\in R$.
On the other hand, we can take some point $z\in\Sum_{r,n}(Y)$ which satisfying $(p^{-1},z)\in R$.
Since $|d^{r,n}_X(p^1,p^{-1})-d^{r,n}_Y(y_k,z)|\leq\dis(R)$,
we have
\[
d^{r,n}_Y(y_k,z)\leq\dis(R)+d^{r,n}_X(p^1,p^{-1})<2r+3r=5r \,,
\]
thus $z\in Y_k$; otherwise we obtain $d^{r,n}_Y(y_k,z)\geq5r$ by the definition of $d^{r,n}_Y$.
Then we have $d^{r,n}_Y(y_k,z)\leq\diam(Y_k)\leq r$, so we obtain a contradiction:
\[
3r=d^{r,n}_X(p^1,p^{-1})\leq\dis(R)+d^{r,n}_Y(y_k,z)<2r+r=3r \,.
\]
Therefore, we have $R[\{p^1\}]\subset\{p^1,p^{-1}\}$.\\

Secondary, we shall show that $R[\{p^1\}]$ is a one point set, which means there exists $R(p^1)\in\{p^1,p^{-1}\}$ such that $R[\{p^1\}]=\{R(p^1)\}$.
Let $z_0,z_1$ be points of $R[\{p^1\}]$.
Suppose $z_0\neq z_1$ for the sake of contradiction.
Then we have $\{z_0,z_1\}=\{p^1,p^{-1}\}$ by $z_0,z_1\in R[\{p^1\}]\subset\{p^1,p^{-1}\}$, so we have
\[
3r
=d^{r,n}_Y(p^1,p^{-1})
=d^{r,n}_Y(z_0,z_1)
\leq\dis(R)+d^{r,n}_X(p^1,p^1)
<2r\,,
\]
which is a contradiction to $r>0$.
As a consequence, we can define $R(p^1)\in\{p^1,p^{-1}\}$ as an above point.
In addition, we can show that $R[\{p^{-1}\}]$ is a one-point set contained in $\{p^1,p^{-1}\}$ similarly, therefore we can also define $R(p^{-1})\in\{p^1,p^{-1}\}$ satisfying $R[\{p^{-1}\}]=\{R(p^{-1})\}$.\\

In particular, if $R(p^1)=R(p^{-1})$, we have
\[
3r=d^{r,n}_X(p^1,p^{-1})\leq\dis(R)+d^{r,n}_Y(R(p^1),R(p^{-1}))<2r\,,
\] 
which is a contradiction.
It means $R(p^1)=-R(p^{-1})$.
Additionally, the above argument can also be applied to $R^{-1}\in\GHC(\Sum_{r,n}(Y),\Sum_{r,n}(X))$.
Hence, we can also define $R^{-1}(p^1)$, $R^{-1}(p^{-1})$ and they satisfy $R^{-1}(p^1)=-R^{-1}(p^{-1})$.\\

We only need to verify $R(p^1)=R^{-1}(p^1)$.
By the argument above, we can consider $R$, $R^{-1}$ as two maps $R,R^{-1}:\{p^1,p^{-1}\}\to\{p^1,p^{-1}\}$, and see two maps $R$ and $R^{-1}$ are inverses of each other by the definition of $R$, $R^{-1}$.
It means the maps are equal to each other, therefore we have $R(p^1)=R^{-1}(p^1)$.
\end{step}

\begin{step}
\label{setcorr}
In this step, we shall show that there exists a permutation $\sigma$ of $\{1,2,\ldots,n\}$
such that 
\[
R[X_k]\subset Y_{\sigma(k)}
\,,\,
R^{-1}[Y_k]\subset X_{\sigma^{-1}(k)}
\]
for an arbitrary number $k\in\{1,2,\ldots,n\}$.\\

For an arbitrary fixed point $x_k\in X_k$, we can show that all points $y\in R[\{x_k\}]$ do not belong to $Y_0$.
Indeed, we have $d^{r,n}_Y(y,R(p^s))\neq0$ for any $s\in\{1,-1\}$ by calculation:
\[
0<3r=5r-2r\leq5rk-2r<d^{r,n}_X(x_k,p^s)-\dis(R)\leq d^{r,n}_Y(y,R(p^s))\,.
\]
It means $R[X_k]\subset \bigcup_{k=1}^nY_k$.\\

Now, we defined a map $\sigma:\{1,2,\ldots,n\}\to\{1,2,\ldots,n\}$ by the graph
\[
\Set{(k,l)}{\mathrm{there\ exist}\ x\in X_k,\ y\in Y_l\ \mathrm{such\ that}\ (x,y)\in R}\,.
\]
To show the well-definedness of $\sigma$,
It suffices to show $l_0=l_1$ for arbitrary $4$ points $x,\xi\in X_k$, $y\in Y_{l_0},$ $\eta\in Y_{l_1}$ satisfying $(x,y),(\xi,\eta)\in R$. 
Since we have
\[
d^{r,n}_Y(y,\eta)\leq\dis(R)+d^{r,n}_X(x,\xi)<2r+r=3r\,,
\]
we can obtain $|l_0-l_1|<1$ if $l_0\neq l_1$.
However $|l_0-l_1|<1$ means $l_0=l_1$, and it contradics to $l_0\neq l_1$.\\

The above argument can also be applied to $R^{-1}$, therefore we can define a map $\tau:\{1,2,\ldots,n\}\to\{1,2,\ldots,n\}$ similarly.
Moreover, we have
\[
R[X_k]\subset Y_{\sigma(k)}
\,,\,
R^{-1}[Y_k]\subset X_{\tau(k)}
\]
for any $k\in\{1,2,\ldots,n\}$ by the definition of $\sigma,\,\tau$.\\

To show the claim of this step, we shall show that $\sigma$ and $\tau$ are inverses of each other. 
For any $k\in\{1,2,\ldots,n\}$, it is straightforward to see:
\[
X_k\subset R^{-1}[R[X_k]]\subset R^{-1}[Y_{\sigma(k)}]\subset X_{\tau(\sigma(k))}.
\]
It means $\tau(\sigma(k))=k$ by the assumption of disjointness.
Reversing the roles of $R$ and $R^{-1}$, we can also obtain $\sigma(\tau(k))=k$.
\end{step}

\begin{step}
In this step, we shall show that $\sigma$ is the identity map.\\

Firstly, we can obtain $\sigma(1)=1$. 
Indeed, we have
\[
5r\sigma(1)=d^{r,n}_Y(y_1,R(p^+))\leq\dis(R)+d^{r,n}_X(x_1,p^+)<2r+5r=7r
\]
for some points $x_1\in X_1$ and $y_1\in R[\{x_1\}]$, therefore $\sigma(1)\leq1$.\\

Secondary, we shall show $\sigma(k)=k$ for $k\in\{2,3,\ldots,n\}$.
Since $\sigma(k)>1$, we have
\[
5r|\sigma(k)-k|=5r||\sigma(k)-1|-|k-1||=d^{r,n}_Y(y_k,y_1)-d^{r,n}_X(x_k,x_1)\leq\dis(R)<2r
\]
for some points $x_k\in X_k$ and $y_k\in R[\{x_k\}]$.
Then we obtain $|\sigma(k)-k|<1$ , that is $\sigma(k)=k$.
\end{step}

\begin{step}
In this step, we shall show a contradiction.
As a result of step \ref{setcorr}, we can define $R_k\in\GHC(X_k,Y_k)$ by $R_k=R\cap(X_k\times Y_k)$ for any $k\in\{1,2,\ldots,n\}$.
In particular, since $R_k\subset R$, we have
\[
2\dGH(X_k,Y_k)\leq\dis(R_k)\leq\dis(R)<2m\,,
\]
that is $2m<2m$.
It is a contradiction.
\end{step}

As a consequence, we have $\dGH(\Sum_{r,n}(X),\Sum_{r,n}(Y))=m$.
\end{proof}

\begin{proof}[Proof of Corollary \ref{theorem2}]
Since $\sup_{X\in\X}\diam(X)<\infty$, $\sup_{Y\in\X}\diam(Y)<\infty$, there exists $r>0$ such taht $\X,\Y\subset\M_{\leq r}$, which means $\X\times\Y\subset\M_{\leq r}\times\M_{\leq r}$.
Therefore $\Sum_{r,2}:\M_{\leq r}\times\M_{\leq r}\to\B(10r)$ induces an isometric embedding we seek.
\end{proof}

\begin{proof}[Proof of Corollary \ref{theorem3}]
By the assumption, $X,Y$ are isometric to some bounded metric subspaces of $\M$
\end{proof}

\section{Acknowledgments}
\*\indent This research was financially supported by the Kansai University Grant-in-Aid for research progress in graduate course, 2024.
The author acknowledges this support.
Dr.Yoshito Ishiki, Professor Toshihiro Shoda, and Supervisor Atsushi Fujioka reviewed his paper and encouraged him.
The author thanks them for their kindness.


\begin{thebibliography}{9}
\bibitem{Banach}
S.~Banach.
\newblock {\em Th\'{e}orie des \'{o}perations lin\'{e}aires}.
\newblock Chelsea Publishing Co., New York, 1955.

\bibitem{Dutrieux}
Y.~Dutrieux and G.~Lancien. 
\newblock{\em Isometric embeddings of compact
spaces into Banach spaces}.
\newblock Journal of Functional Analysis, {\bf 255}, 2008, pp.494--501.

\bibitem{Frechet}
M.~Fr\'{e}chet.
\newblock {\em Les dimensions d'un ensemble abstrait}.
\newblock Math. Ann. {\bf 68}, 1909--10, pp.145--168.

\bibitem{Tuzhilin1}
S.~Iliadis, A.~Ivanov and A.~Tuzhilin.
\newblock {\em Local structure of Gromov-Hausdorff space, and isometric embeddings of finite metric spaces into this space}, 
\newblock Topology and its Applications, {\bf 221}, 2017, pp.393--398. 
DOI: 10.1016/j.topol.2017.02.050.

\bibitem{Tuzhilin2}
A.~Ivanov and A.~Tuzhilin.
\newblock {\em Isometry group of Gromov–Hausdorff space}
\newblock Matematicki Vesnik, {\bf 71} (1--2), 2019, pp.123--154.

\bibitem{Urysohn}
P.~Urysohn.
\newblock {\em Sur un espace m\'{e}trique universel}.
\newblock Bull. Sci. Math. {\bf 51}, 1927, pp.43--64 and 74--90.

\end{thebibliography}
\end{document}